\documentclass{amsart}
\usepackage{amsmath}
\usepackage{amsfonts}

\newtheorem{theorem}{Theorem}
\theoremstyle{plain}

\newtheorem{definition}{Definition}

\newtheorem{proposition}{Proposition}

\numberwithin{equation}{section}

\begin{document}
\title[]{On Pseudo-Hermitian Magnetic Curves \\
in Sasakian Manifolds}
\author{\c{S}aban G\"{u}ven\c{c}}
\address[\c{S}. G\"{u}ven\c{c} and C. \"{O}zg\"{u}r]{Balikesir University,
Department of Mathematics\\
Campus of Cagis, 10145, Balikesir, TURKEY}
\email[\c{S}. G\"{u}ven\c{c}]{sguvenc@balikesir.edu.tr}
\author{Cihan \"{O}zg\"{u}r}
\email[C.~\"{O}zg\"{u}r]{cozgur@balikesir.edu.tr}

\begin{abstract}
We define pseudo-Hermitian magnetic curves in Sasakian manifolds endowed
with the Tanaka-Webster connection. After we give a complete classification
theorem, we construct parametrizations of pseudo-Hermitian magnetic curves
in $%
\mathbb{R}
^{2n+1}(-3)$.
\end{abstract}

\maketitle

\section{Introduction}

The study of the motion of a charged particle in a constant and
time-independent static magnetic field on a Riemannian surface is known as
the Landau--Hall problem. The main problem is to study the movement of a
charged particle moving in the Euclidean plane $\mathbb{E}^{2}$. The
solution of the Lorentz equation (called also the Newton equation)
corresponds to the motion of the particle. The trajectory of a charged
particle moving on a Riemannian manifold under the action of the magnetic
field is a very interesting problem from a geometric point of view.

Let $\left( N,g\right) $ be a Riemannian manifold, and $F$ a closed $2$%
-form, $\Phi $ the Lorentz force, which is a $(1,1)$-type tensor field, on $N
$.\ $F$ is called a\textit{\ magnetic field }if it is associated to $\Phi $
by the relation%
\begin{equation}
F(X,Y)=g(\Phi X,Y),\text{\ }  \label{F}
\end{equation}%
where $X$ and $Y$ are vector fields on $N$ (see \cite{Adachi-1996}, \ \cite%
{BRCF} and \cite{Comtet-1987}). Let $\nabla $ be the Riemannian connection
on $N$ and consider a differentiable curve $\alpha $ $:I\rightarrow N$,
where $I$ denotes an open interval of $%
\mathbb{R}
$. $\alpha $ is said to be a \textit{magnetic curve }for the magnetic field $%
F,$ if it is a solution of the Lorentz equation given by%
\begin{equation}
\nabla _{\alpha ^{\prime }(t)}\alpha ^{\prime }(t)=\Phi (\alpha ^{\prime
}(t)).  \label{Lorentz eq}
\end{equation}%
From the definition of magnetic curves, it is straightforward to see that
their speed is constant. Specifically, unit-speed magnetic curves are called 
\textit{normal magnetic curves }\cite{DIMN-2015}.

In \cite{DIMN-2015}, Dru\c{t}\u{a}-Romaniuc, Inoguchi, Munteanu and Nistor
magnetic curves studied in a Sasakian manifold. \ Magnetic curves in
cosymplectic manifolds were studied in \cite{DIMN-2016} by the same authors.
In \cite{IM-2017}, \ $3$-dimensional Berger spheres and their magnetic
curves were considered by Inoguchi and Munteanu. Magnetic trajectories of an
almost contact metric manifold were studied in \cite{JMN-2015}, \ by Jleli,
Munteanu and Nistor. The classification of all uniform magnetic trajectories
of a charged particle moving on a surface under the action of a uniform
magnetic field was obtained in \cite{Munteanu-2014}, by Munteanu.
Furthermore, normal magnetic curves in para-Kaehler manifolds were
researched in \cite{JM-2015}, by Jleli and Munteanu. In \cite{MN-2012},
Munteanu and Nistor obtained the classification of complete classification
of unit-speed Killing magnetic curves in $\mathbb{S}^{2}\times \mathbb{R}$.
Moreover, in \cite{MN-2014}, they studied magnetic curves on $\mathbb{S}%
^{2n+1}$. $3$-dimensional normal para-contact metric manifolds and their
magnetic curves of a Killing vector field were investigated in \cite%
{CMP-2015}, by Calvaruso, Munteanu and Perrone. The second author gave the
parametric equations of all normal magnetic curves in the $3$-dimensional
Heisenberg group in \cite{Ozgur-2017}. Recently, the present authors
considered slant magnetic curves in $S$-manifolds in \cite{GO-2019}.

These studies motivate us to investigate pseudo-Hermitian magnetic curves in 
$(2n+1)$-dimensional Sasakian manifolds endowed with the Tanaka-Webster
connection. In Section \ref{Preliminaries}, we summarize the fundamental
definitions and properties of Sasakian manifolds and the unique connection,
namely the Tanaka-Webster connection. We give the main classification
theorems for pseudo-Hermitian magnetic curves in Section \ref{magnetic}. We
show that a pseudo-Hermitian magnetic curve can not have osculating order
bigger than $3$. In the last section, after a brief information on $%
\mathbb{R}
^{2n+1}(-3)$, we obtain the parametric equations of pseudo-Hermitian
magnetic curves in $%
\mathbb{R}
^{2n+1}(-3)$ endowed with the Tanaka-Webster connection.

\section{Preliminaries\label{Preliminaries}}

Let $N$ be a ($2n+1$)-dimensional Riemannian manifold satisfying the
following equations%
\begin{equation}
\begin{array}{cccc}
\phi ^{2}(X)=-X+\eta (X)\xi , & \eta (\xi )=1, & \phi (\xi )=0, & \eta \circ
\phi =0,%
\end{array}
\label{eq1}
\end{equation}%
\begin{equation}
\begin{array}{cc}
g(X,\xi )=\eta (X), & g(X,Y)=g(\phi X,\phi Y)+\eta (X)\eta (Y),%
\end{array}
\label{eq2}
\end{equation}%
for all vector fields $X,Y$ on $N$, where $\phi $ is a ($1,1$)-type tensor
field, $\eta $ is a $1$-form, $\xi $ is a vector field and $g$ is a
Riemannian metric on $N$. In this case, $\left( N,\phi ,\xi ,\eta ,g\right) $
is said to be an \textit{almost contact metric manifold} \cite{Blair}.
Moreover, if $d\eta (X,Y)=\Phi (X,Y)$, where $\Phi (X,Y)=g(X,\phi Y)$ is the 
\textit{fundamental }$2$\textit{-form} of the manifold, then $N$ is said to
be a \textit{contact metric manifold} \cite{Blair}.

Furthermore, if we denote the the Nijenhuis torsion of $\phi $ by $[\phi
,\phi ]$, for all $X,Y$ $\in \chi (N)$, the condition given by%
\begin{equation*}
\lbrack \phi ,\phi ](X,Y)=-2d\eta (X,Y)\xi ,
\end{equation*}%
is called the \textit{normality condition} of the almost contact metric
structure. An almost contact metric manifold turns into a \textit{Sasakian
manifold }if the normality condition is satisfied\textit{\ }\cite{Blair}.

From Lie differentiation operator in the characteristic direction $\xi ,$
the operator $h$ is defined by 
\begin{equation*}
h=\frac{1}{2}L_{\xi }\phi .
\end{equation*}%
It is directly found that the structural operator $h$ is symmetric. It also
validates the equations below, where we donote the Levi-Civita connection by 
$\nabla $:%
\begin{equation}
\begin{array}{cc}
h\xi =0, & h\phi =-\phi h,%
\end{array}
\label{structural}
\end{equation}%
(see \cite{Blair}).

If we denote the Tanaka-Webster connection on $N$ by $\widehat{\nabla }$ (%
\cite{Tanaka}, \cite{Webster}) , then we have%
\begin{equation*}
\widehat{\nabla }_{X}Y=\nabla _{X}Y+\eta (X)\phi Y+(\widehat{\nabla }%
_{X}\eta )(Y)\xi -\eta (Y)\nabla _{X}\xi 
\end{equation*}%
for all vector fields $X,Y$ on $N$. By the use of equations (\ref{structural}%
), the Tanaka-Webster connection can be calculated as%
\begin{equation}
\widehat{\nabla }_{X}Y=\nabla _{X}Y+\eta (X)\phi Y+\eta (Y)(\phi X+\phi
hX)-g(\phi X+\phi hX,Y)\xi .  \label{tanaka}
\end{equation}%
The torsion of the Tanaka-Webster connection is%
\begin{equation}
\widehat{T}(X,Y)=2g(X,\phi Y)\xi +\eta (Y)\phi hX-\eta (X)\phi hY.
\label{torsion}
\end{equation}%
In a Sasakian manifold, from the fact that $h=0$ (see \cite{Blair}),
equations (\ref{tanaka}) and (\ref{torsion}) can be rewritten as:%
\begin{equation}
\widehat{\nabla }_{X}Y=\nabla _{X}Y+\eta (X)\phi Y+\eta (Y)\phi X-g(\phi
X,Y)\xi ,\text{\ \ }  \label{tanaka-webster}
\end{equation}%
\begin{equation*}
\widehat{T}(X,Y)=2g(X,\phi Y)\xi .
\end{equation*}%
The following proposition states why the Tanaka-Webster connection is unique:

\begin{proposition}
\cite{Tanno-89} The Tanaka-Webster connection on a contact Riemannian
manifold $N=(N,\phi ,\xi ,\eta ,g)$ is the unique linear connection
satisfying the following four conditions:

(a) $\ \ \ \widehat{\nabla }\eta =0$, $\widehat{\nabla }\xi =0;$

(b) \ \ \ $\widehat{\nabla }g=0$, $\widehat{\nabla }\phi =0;$

(c) \ \ $\widehat{T}(X,Y)=-\eta (\left[ X,Y\right] )\xi $, $\ \ \forall
X,Y\in D;$

(d) \ \ $\widehat{T}(\xi ,\phi Y)=-\phi \widehat{T}(\xi ,Y),$ \ $\forall
Y\in D.$
\end{proposition}

\bigskip

\section{Magnetic Curves with respect to the Tanaka-Webster Connection\label%
{magnetic}}

Let $\left( N,\phi ,\xi ,\eta ,g\right) $ be an $n$-dimensional Riemannian
manifold and $\alpha :I\rightarrow N$ a curve parametrized by arc-length. If
there exists $g$-orthonormal vector fields $E_{1},E_{2},...,E_{r}$ along $%
\alpha $ such that 
\begin{eqnarray}
E_{1} &=&\alpha ^{\prime },  \notag \\
\widehat{\nabla }_{E_{1}}E_{1} &=&\widehat{k}_{1}E_{2},  \notag \\
\widehat{\nabla }_{E_{1}}E_{2} &=&-\widehat{k}_{1}E_{1}+\widehat{k}_{2}E_{3},
\label{Frenetequations} \\
&&...  \notag \\
\widehat{\nabla }_{E_{1}}E_{r} &=&-\widehat{k}_{r-1}E_{r-1},  \notag
\end{eqnarray}%
then $\alpha $ is called a \textit{Frenet curve for }$\widehat{\nabla }$ 
\textit{of osculating order }$r$ , $\left( 1\leq r\leq n\right) $. Here $%
\widehat{k}_{1},...,\widehat{k}_{r-1}$ are called \textit{pseudo-Hermitian
curvature functions of }$\alpha $ and these functions are positive valued on 
$I$. A \textit{geodesic} \textit{for }$\widehat{\nabla }$ (or \textit{%
pseudo-Hermitian geodesic}) is a Frenet curve of osculating order $1$ for $%
\widehat{\nabla }$. If $r=2$ and a $\widehat{k}_{1}$ is a constant, then $%
\alpha $ is called a \textit{pseudo-Hermitian circle.} A \textit{%
pseudo-Hermitian helix of order }$r$, $r\geq 3$, is a Frenet curve for $%
\widehat{\nabla }$ of osculating order $r$ with non-zero positive constant
pseudo-Hermitian curvatures $\widehat{k}_{1},...,\widehat{k}_{r-1}$. If we
shortly state \textit{pseudo-Hermitian helix,} we mean its osculating order
is $3$ \cite{CL}.

Let $N=\left( N^{2n+1},\phi ,\xi ,\eta ,g\right) $ be a Sasakian manifold
endowed with the Tanaka-Webster connection $\widehat{\nabla }.$ Let us
denote the the fundamental $2$-form of $N$ by $\Omega $. Then, we have%
\begin{equation}
\Omega (X.Y)=g(X,\phi Y),  \label{omega}
\end{equation}%
(see \cite{Blair}). From the fact that $N$ is a Sasakian manifold, we have $%
\Omega =d\eta $. Hence, $d\Omega =0$, i.e., it is closed. Thus, we can
define a magnetic field $F_{q}$ on $N$  by 
\begin{equation*}
F_{q}(X,Y)=q\Omega (X.Y),
\end{equation*}%
namely the \textit{contact magnetic field with strength} $q$, where $X,Y\in
\chi (N)$ and $q\in 
\mathbb{R}
$ \cite{JMN-2015}. We will assume that $q\neq 0$ to avoid the absence of the
strength of magnetic field (see \cite{CFG-2009} and \cite{DIMN-2015}).

From (\ref{F}) and (\ref{omega}), the Lorentz force $\Phi $ associated to
the contact magnetic field $F_{q}$ can be written as%
\begin{equation*}
\Phi =-q\phi .
\end{equation*}%
So the Lorentz equation (\ref{Lorentz eq}) is%
\begin{equation}
\nabla _{E_{1}}E_{1}=-q\phi E_{1},  \label{magneticLC}
\end{equation}%
where $\alpha :I\rightarrow N$ is a curve with arc-length parameter, $%
E_{1}=\alpha ^{\prime }$ is the tangent vector field and $\nabla $ is the
Levi-Civita connection (see \cite{DIMN-2015} and \cite{JMN-2015}). By the
use of equations (\ref{tanaka-webster}) and (\ref{magneticLC}), we have%
\begin{equation}
\widehat{\nabla }_{E_{1}}E_{1}=\left[ -q+2\eta (E_{1})\right] \phi E_{1}.
\label{lorentz-tw}
\end{equation}

\begin{definition}
Let $\alpha :I\rightarrow N$ be a unit-speed curve in a Sasakian manifold $%
N=\left( N^{2n+1},\phi ,\xi ,\eta ,g\right) $ endowed with the
Tanaka-Webster connection $\widehat{\nabla }$. Then it is called a \textit{%
normal magnetic curve with respect to} \textit{the Tanaka-Webster connection}
$\widehat{\nabla }$ (or shortly \textit{pseudo-Hermitian magnetic}) if it
satisfies equation (\ref{lorentz-tw}).
\end{definition}

If $\eta (E_{1})=\cos \theta $ is a constant, then $\alpha $ is called 
\textit{slant} \cite{CIL}. From the definition of pseudo-Hermitian magnetic
curves, we have the following direct result as in the Levi-Civita case:

\begin{proposition}
\label{prop-slant}If $\alpha $ is pseudo-Hermitian magnetic curve in a
Sasakian manifold, then it is slant.
\end{proposition}

\begin{proof}
Let $\alpha :I\rightarrow N$ be pseudo-Hermitian magnetic. Then, we find%
\begin{eqnarray*}
\frac{d}{dt}g(E_{1},\xi ) &=&g(\widehat{\nabla }_{E_{1}}E_{1},\xi )+g(E_{1},%
\widehat{\nabla }_{E_{1}}\xi ) \\
&=&g(\left[ -q+2\eta (E_{1})\right] \phi E_{1},\xi ) \\
&=&0.
\end{eqnarray*}%
So we obtain%
\begin{equation*}
\eta (E_{1})=\cos \theta =constant,
\end{equation*}%
which completes the proof.
\end{proof}

As a result, we can rewrite equation (\ref{lorentz-tw}) as%
\begin{equation}
\widehat{\nabla }_{E_{1}}E_{1}=\left( -q+2\cos \theta \right) \phi E_{1},
\label{lorentz-tw2}
\end{equation}%
where $\theta $ is the contact angle of $\alpha $. Now, we can state the
following theorem:

\begin{theorem}
\label{Theorem1}Let $\left( N^{2n+1},\phi ,\xi ,\eta ,g\right) $ be a
Sasakian manifold endowed with the Tanaka-Webster connection $\widehat{%
\nabla }$. Then $\alpha :I\rightarrow N$ is pseudo-Hermitian magnetic if and
only if it belongs to the following list:

(a) pseudo-Hermitian non-Legendre slant geodesics (including
pseudo-Hermitian geodesics as integral curves of $\xi $);

(b) pseudo-Hermitian Legendre circles with $\widehat{k}_{1}=\left\vert
q\right\vert $ and having the Frenet frame field (for $\widehat{\nabla }$)%
\begin{equation*}
\left\{ E_{1},-sgn(q)\phi E_{1}\right\} ;
\end{equation*}

(c) pseudo-Hermitian slant helices with 
\begin{equation*}
\widehat{k}_{1}=\left\vert -q+2\cos \theta \right\vert \sin \theta ,\text{ }%
\widehat{k}_{2}=\left\vert -q+2\cos \theta \right\vert \varepsilon \cos
\theta 
\end{equation*}%
and having the Frenet frame field (for $\widehat{\nabla }$)%
\begin{equation*}
\left\{ E_{1},\frac{\delta }{\sin \theta }\phi E_{1},\frac{\varepsilon }{%
\sin \theta }\left( \xi -\cos \theta E_{1}\right) \right\} ,
\end{equation*}%
where $\delta =sgn(-q+2\cos \theta )$, $\varepsilon =sgn(\cos \theta )$ and $%
\cos \theta \neq \frac{q}{2}$.
\end{theorem}

\begin{proof}
Let us assume that $\alpha :I\rightarrow N$ is a normal magnetic curve with
respect to $\widehat{\nabla }$. Consequently, equation (\ref{lorentz-tw2})
must be validated. Let us assume $\widehat{k}_{1}=0$. Hence, we have $\cos
\theta =\frac{q}{2}$ or $\phi E_{1}=0$. If $\cos \theta =\frac{q}{2}$, then $%
\alpha $ is a pseudo-Hermitian non-Legendre slant geodesic. Otherwise, $\phi
E_{1}=0$ gives us $E_{1}=\pm \xi $. Thus, $\alpha $ is a pseudo-Hermitian
geodesic as an integral curve of $\pm \xi $. So we have just proved that $%
\alpha $ belongs to (a) from the list, if the osculating order $r=1$. Now,
let $\widehat{k}_{1}\neq 0$. From equation (\ref{lorentz-tw2}) and the
Frenet equations for $\widehat{\nabla }$, we find 
\begin{equation}
\widehat{\nabla }_{E_{1}}E_{1}=\widehat{k}_{1}E_{2}=\left( -q+2\cos \theta
\right) \phi E_{1}  \label{3.5}
\end{equation}%
Since $E_{1}$ is unit, equation (\ref{eq2}) gives us 
\begin{equation}
g(\phi E_{1},\phi E_{1})=\sin ^{2}\theta .  \label{fi}
\end{equation}%
By the use of (\ref{3.5}) and (\ref{fi}), we obtain%
\begin{equation}
\widehat{k}_{1}=\left\vert -q+2\cos \theta \right\vert \sin \theta ,
\label{3.6}
\end{equation}%
which is a constant. Let us denote $\delta =sgn(-q+2\cos \theta )$. From (%
\ref{3.6}), we can write%
\begin{equation}
\phi E_{1}=\delta \sin \theta E_{2}.  \label{3.7}
\end{equation}%
Let us assume $\widehat{k}_{2}=0$, that is, $r=2$. From the fact that $%
\widehat{k}_{1}$ is a constant, $\alpha $ is a pseudo-Hermitian circle. (\ref%
{3.7}) gives us%
\begin{equation*}
\eta \left( \phi E_{1}\right) =0=\delta \sin \theta \eta \left( E_{2}\right)
,
\end{equation*}%
which is equivalent to%
\begin{equation*}
\eta \left( E_{2}\right) =0.
\end{equation*}%
Differentiating this last equation with respect to $\widehat{\nabla }$, we
obtain 
\begin{equation*}
\widehat{\nabla }_{E_{1}}\eta \left( E_{2}\right) =0=g\left( \widehat{\nabla 
}_{E_{1}}E_{2},\xi \right) +g\left( E_{2},\widehat{\nabla }_{E_{1}}\xi
\right) .
\end{equation*}%
Since $\widehat{\nabla }\xi =0$ and $r=2$, we have%
\begin{equation*}
g(-\widehat{k}_{1}E_{1},\xi )=0,
\end{equation*}%
that is, $\eta (E_{1})=0$. Hence, $\alpha $ is Legendre and $\cos \theta =0$%
. From equation (\ref{3.6}), we get $\widehat{k}_{1}=\left\vert q\right\vert 
$. In this case, we also obtain $\delta =-sgn(q)$ and $E_{2}=-sgn(q)\phi
E_{1}$. We have proved that $\alpha $ belongs to (b) from the list, if the
osculating order $r=2$. Now, let us assume $\widehat{k}_{2}\neq 0$. If we
use $\widehat{\nabla }\phi =0$, we calculate 
\begin{equation}
\widehat{\nabla }_{E_{1}}\phi E_{1}=\widehat{k}_{1}\phi E_{2}.  \label{***}
\end{equation}%
From (\ref{eq1}) and (\ref{3.7}), we find%
\begin{equation}
\phi ^{2}E_{1}=-E_{1}+\cos \theta \xi =\delta \sin \theta \phi E_{2},
\label{3.8}
\end{equation}%
which gives us%
\begin{equation*}
\phi E_{2}=\frac{\delta }{\sin \theta }\left( -E_{1}+\cot \theta \xi \right)
.
\end{equation*}%
So equation (\ref{***}) becomes%
\begin{equation}
\widehat{\nabla }_{E_{1}}\phi E_{1}=\widehat{k}_{1}\frac{\delta }{\sin
\theta }\left( -E_{1}+\cot \theta \xi \right) .  \label{3.10}
\end{equation}%
If we differentiate equation (\ref{3.7}) with respect to $\widehat{\nabla }$%
, we also have%
\begin{eqnarray}
\widehat{\nabla }_{E_{1}}\phi E_{1} &=&\delta \sin \theta \widehat{\nabla }%
_{E_{1}}E_{2}  \label{3.11} \\
&=&\delta \sin \theta \left( -\widehat{k}_{1}E_{1}+\widehat{k}%
_{2}E_{3}\right) .  \notag
\end{eqnarray}%
By the use of (\ref{3.10}) and (\ref{3.11}), we obtain%
\begin{equation}
\widehat{k}_{1}\cot \theta \left( \xi -\cos \theta E_{1}\right) =\widehat{k}%
_{2}\sin \theta E_{3}.  \label{3.12}
\end{equation}%
One can easily see that%
\begin{equation*}
g(\xi -\cos \theta E_{1},\xi -\cos \theta E_{1})=\sin ^{2}\theta .
\end{equation*}%
From (\ref{3.12}), we calculate%
\begin{equation*}
\widehat{k}_{2}=\left\vert -q+2\cos \theta \right\vert \varepsilon \cos
\theta ,
\end{equation*}%
where we denote $\varepsilon =sgn(\cos \theta ).$ As a result, we get%
\begin{equation}
E_{3}=\frac{\varepsilon }{\sin \theta }\left( \xi -\cos \theta E_{1}\right) ,
\label{*}
\end{equation}%
\begin{equation*}
E_{2}=\frac{\delta }{\sin \theta }\phi E_{1}.
\end{equation*}%
If we differentiate (\ref{*}) with respect to $\widehat{\nabla }$, since $%
\phi E_{1}\parallel E_{2}$, we find $\widehat{k}_{3}=0$. So we have just
completed the proof of (c). Considering the fact that $\widehat{k}_{3}=0$,
the Gram-Schmidt process ends. Thus, the list is complete.

Conversely, let $\alpha :I\rightarrow N$ belong to the given list. It is
easy to show that equation (\ref{lorentz-tw2}) is satisfied. Hence, $\alpha $
is pseudo-Hermitian magnetic.
\end{proof}

A pseudo-Hermitian geodesic is said to be a pseudo-Hermitian $\phi $-curve
if $sp\left\{ E_{1},\phi E_{1},\xi \right\} $ is $\phi $-invariant. A Frenet
curve of osculating order $r=2$ is said to be a pseudo-Hermitian $\phi $%
-curve if $sp\left\{ E_{1},E_{2},\xi \right\} $ is $\phi $-invariant. A
Frenet curve of osculating order $r\geq 3$ is said to be a pseudo-Hermitian $%
\phi $-curve if $sp\left\{ E_{1},E_{2},...,E_{r}\right\} $ is $\phi $%
-invariant.

\begin{theorem}
Let $\alpha :I\rightarrow N$ be a pseudo-Hermitian $\phi $-helix of order $%
r\leq 3$, where $N=\left( N^{2n+1},\phi ,\xi ,\eta ,g\right) $ is a Sasakian
manifold endowed with the Tanaka-Webster connection $\widehat{\nabla }$.
Then:

(a) If $\cos \theta =\pm 1$, then it is an integral curve of $\xi $, i.e. a
pseudo-Hermitian geodesic and it is a pseudo-Hermitian magnetic curve for $%
F_{q}$ for arbitrary $q$;

(b) If $\cos \theta \notin \left\{ -1,0,1\right\} $ and $\widehat{k}_{1}=0$,
then it is a pseudo-Hermitian non-Legendre slant geodesic and it is a
pseudo-Hermitian magnetic curve for $F_{2\cos \theta };$

(c) If $\cos \theta =0$ and $\widehat{k}_{1}\neq 0$, i.e. $\alpha $ is a
Legendre $\phi $-curve, then it is a pseudo-Hermitian magnetic circle
generated by $F_{-\delta \widehat{k}_{1}}$, where $\delta =sgn(g(\phi
E_{1},E_{2}))$;

(d) If $\cos \theta =\frac{\varepsilon \widehat{k}_{2}}{\sqrt{\widehat{k}%
_{1}^{2}+\widehat{k}_{2}^{2}}}$ and $\widehat{k}_{2}\neq 0$, then it is a
magnetic curve for $F_{-\delta \sqrt{\widehat{k}_{1}^{2}+\widehat{k}_{2}^{2}}%
+\frac{2\varepsilon \widehat{k}_{2}}{\sqrt{\widehat{k}_{1}^{2}+\widehat{k}%
_{2}^{2}}}}$ with respect to $\widehat{\nabla }$, where $\delta =sgn(g(\phi
E_{1},E_{2}))$\ and $\varepsilon =sgn(\cos \theta )$.

(e) Except above cases, $\alpha $ cannot be a pseudo-Hermitian magnetic
curve for any $F_{q}$.
\end{theorem}

\begin{proof}
Firstly, let us assume $\cos \theta =\pm 1$, that is, $E_{1}=\pm \xi $. As a
result, we have 
\begin{equation*}
\widehat{\nabla }_{E_{1}}E_{1}=0,\text{ }\phi E_{1}=0.
\end{equation*}%
Hence, equation (\ref{lorentz-tw2}) is satisfied for arbitrary $q$. This
proves $(a)$. Now, let us take $\cos \theta \notin \left\{ -1,0,1\right\} $
and $\widehat{k}_{1}=0$. In this case, we obtain%
\begin{equation*}
\widehat{\nabla }_{E_{1}}E_{1}=0,\text{ }\phi E_{1}\neq 0.
\end{equation*}%
So equation (\ref{lorentz-tw2}) is valid for $q=2\cos \theta $. The proof of 
$(b)$ is over. Next, let us assume $\cos \theta =0$ and $\widehat{k}_{1}\neq
0$. One can easily see that $\alpha $ has the Frenet frame field (for $%
\widehat{\nabla }$) 
\begin{equation*}
\left\{ E_{1},\delta \phi E_{1}\right\} 
\end{equation*}%
where $\delta $ corresponds to the sign of $g(\phi E_{1},E_{2})$.
Consequently, we get%
\begin{equation*}
\widehat{\nabla }_{E_{1}}E_{1}=\delta \widehat{k}_{1}\phi E_{1},
\end{equation*}%
that is, $\alpha $ is a pseudo-Hermitian magnetic for $q=-\delta \widehat{k}%
_{1}$. We have just proven $(c)$. Finally, let $\cos \theta =\frac{%
\varepsilon \widehat{k}_{2}}{\sqrt{\widehat{k}_{1}^{2}+\widehat{k}_{2}^{2}}}$
and $\widehat{k}_{2}\neq 0$. So $\alpha $ has the Frenet frame field (for $%
\widehat{\nabla }$)%
\begin{equation*}
\left\{ E_{1},\frac{\delta }{\sin \theta }\phi E_{1},\frac{\varepsilon }{%
\sin \theta }\left( \xi -\cos \theta E_{1}\right) \right\} ,
\end{equation*}%
where $\delta =sgn(g(\phi E_{1},E_{2}))$\ and $\varepsilon =sgn(\cos \theta )
$. After calculations, it is easy to show that equation (\ref{lorentz-tw2})
is satisfied for $q=-\delta \sqrt{\widehat{k}_{1}^{2}+\widehat{k}_{2}^{2}}+%
\frac{2\varepsilon \widehat{k}_{2}}{\sqrt{\widehat{k}_{1}^{2}+\widehat{k}%
_{2}^{2}}}$. Hence, the proof of $(d)$ is completed. Except above cases,
from Theorem \ref{Theorem1}, $\alpha $ cannot be a pseudo-Hermitian magnetic
curve for any $F_{q}$.
\end{proof}

\section{Parametrizations of pseudo-Hermitian magnetic curves in $%
\mathbb{R}
^{2n+1}(-3)$\label{param}}

In this section, our aim is to obtain parametrizations of pseudo-Hermitian
magnetic curves in $%
\mathbb{R}
^{2n+1}(-3)$. To do this, we need to recall some notions from \cite{Blair}.
Let $N=%
\mathbb{R}
^{2n+1}$. Let us denote the coordinate functions of $N$ with $\left(
x_{1},...,x_{n},y_{1},...,y_{n},z\right) $. One may define a structure on $N$
by $\eta =\frac{1}{2}(dz-\overset{n}{\underset{i=1}{\sum }}y_{i}dx_{i})$,
which is a contact structure, since $\eta \wedge (d\eta )^{n}\neq 0$. This
contact structure has the characteristic vector field $\xi =2\frac{\partial 
}{\partial z}$. Let us also consider a $(1,1)$-type tensor field $\phi $
given by the matrix form as%
\begin{equation*}
\phi =\left[ 
\begin{array}{ccc}
0 & \delta _{ij} & 0 \\ 
-\delta _{ij} & 0 & 0 \\ 
0 & y_{j} & 0%
\end{array}%
\right] .
\end{equation*}%
Finally, let us take the Riemannian metric on $N$ given by $g=\eta \otimes
\eta +\frac{1}{4}\overset{n}{\underset{i=1}{\sum }}%
((dx_{i})^{2}+(dy_{i})^{2})$. It is known that $(N,\phi ,\xi ,\eta ,g)$ is a
Sasakian space form and its $\phi $-sectional curvature is $c=-3$. This
special Sasakian space form is denoted by $\mathbb{R}^{2n+1}(-3)$ \cite%
{Blair}. One can easily show that the vector fields 
\begin{equation}
X_{i}=2\frac{\partial }{\partial y_{i}},\text{ }X_{n+i}=\phi X_{i}=2(\frac{%
\partial }{\partial x_{i}}+y_{i}\frac{\partial }{\partial z}),\text{ }i=%
\overline{1,n},\text{ }\xi =2\frac{\partial }{\partial z}  \label{b5}
\end{equation}%
are $g$-unit and $g$-orthogonal. Hence, they form a $g$-orthonormal basis 
\cite{Blair}. Using this basis, the Levi-Civita connection of $\mathbb{R}%
^{2n+1}(-3)$ can be obtained as%
\begin{equation*}
\nabla _{X_{i}}X_{j}=\nabla _{X_{m+i}}X_{m+j}=0,\nabla
_{X_{i}}X_{m+j}=\delta _{ij}\xi ,\nabla _{X_{m+i}}X_{j}=-\delta _{ij}\xi ,
\end{equation*}%
\begin{equation*}
\nabla _{X_{i}}\xi =\nabla _{\xi }X_{i}=-X_{m+i},\nabla _{X_{m+i}}\xi
=\nabla _{\xi }X_{m+i}=X_{i},
\end{equation*}%
(see \cite{Blair}). As a result, the Tanaka-Webster connection of $\mathbb{R}%
^{2n+1}(-3)$ is 
\begin{eqnarray*}
\widehat{\nabla }_{X_{i}}X_{j} &=&\widehat{\nabla }_{X_{m+i}}X_{m+j}=%
\widehat{\nabla }_{X_{i}}X_{m+j}=\widehat{\nabla }_{X_{m+i}}X_{j}= \\
\widehat{\nabla }_{X_{i}}\xi  &=&\widehat{\nabla }_{\xi }X_{i}=\widehat{%
\nabla }_{X_{m+i}}\xi =\widehat{\nabla }_{\xi }X_{m+i}=0,
\end{eqnarray*}%
which was calculated in \cite{Guvenc}. Now, we can investigate the
parametric equations of pseudo-Hermitian magnetic curves in $\mathbb{R}%
^{2n+1}(-3)$ endowed with the Tanaka-Webster connection.

Let $N=\mathbb{R}^{2n+1}(-3)$ endowed with the Tanaka-Webster connection $%
\widehat{\nabla }$. Let $\alpha :I\subseteq 
\mathbb{R}
\rightarrow N,$ $\alpha =\left( \alpha _{1},\alpha _{2},...,\alpha
_{n},\alpha _{n+1},...,\alpha _{2n},\alpha _{2n+1}\right) $ be a
pseudo-Hermitian magnetic curve. Then, the tangential vector field of $%
\alpha $ can be written as%
\begin{equation*}
E_{1}=\sum_{i=1}^{n}\alpha _{i}^{\prime }\frac{\partial }{\partial x_{i}}%
+\sum_{i=1}^{n}\alpha _{n+i}^{\prime }\frac{\partial }{\partial y_{i}}%
+\alpha _{2n+1}^{\prime }\frac{\partial }{\partial z}.
\end{equation*}%
In terms of the $g$-orthonormal basis, $E_{1}$ is rewritten as%
\begin{equation*}
E_{1}=\frac{1}{2}\left[ \sum_{i=1}^{n}\alpha _{n+i}^{\prime
}X_{i}+\sum_{i=1}^{n}\alpha _{i}^{\prime }X_{n+i}+\left( \alpha
_{2n+1}^{\prime }-\sum_{i=1}^{n}\alpha _{i}^{\prime }\alpha _{n+i}\right)
\xi \right] .
\end{equation*}%
From Proposition \ref{prop-slant}, $\alpha $ is slant. Hence, we have 
\begin{equation*}
\eta (E_{1})=\cos \theta =constant,
\end{equation*}%
which is equivalent to%
\begin{equation}
\alpha _{2n+1}^{\prime }=2\cos \theta +\sum_{i=1}^{n}\alpha _{i}^{\prime
}\alpha _{n+i}.  \label{slantcondition}
\end{equation}%
From the fact that $\alpha $ is parametrized with arc-length, we also have%
\begin{equation*}
g(E_{1},E_{1})=1,
\end{equation*}%
that is,%
\begin{equation}
\sum_{i=1}^{2n}\left( \alpha _{i}^{\prime }\right) ^{2}=4\sin ^{2}\theta .
\label{unitcondition}
\end{equation}%
Differentiating $E_{1}$ with respect to $\widehat{\nabla }$, we obtain 
\begin{equation*}
\widehat{\nabla }_{E_{1}}E_{1}=\frac{1}{2}\left( \sum_{i=1}^{n}\alpha
_{n+i}^{\prime \prime }X_{i}+\sum_{i=1}^{n}\alpha _{i}^{\prime \prime
}X_{n+i}\right) .
\end{equation*}%
We also easily find%
\begin{equation*}
\phi E_{1}=\frac{1}{2}\left( -\sum_{i=1}^{n}\alpha _{i}^{\prime
}X_{i}+\sum_{i=1}^{n}\alpha _{n+i}^{\prime }X_{n+i}\right) .
\end{equation*}%
Since $\alpha $ is pseudo-Hermitian magnetic, it must satisfy%
\begin{equation*}
\widehat{\nabla }_{E_{1}}E_{1}=\left( -q+2\cos \theta \right) \phi E_{1}.
\end{equation*}%
Then, we can write%
\begin{equation*}
\frac{\alpha _{n+1}^{\prime \prime }}{-\alpha _{1}^{\prime }}=...=\frac{%
\alpha _{2n}^{\prime \prime }}{-\alpha _{n}^{\prime }}=\frac{\alpha
_{1}^{\prime \prime }}{\alpha _{n+1}^{\prime }}=...=\frac{\alpha
_{n}^{\prime \prime }}{\alpha _{2n}^{\prime }}=-q+2\cos \theta =-\lambda ,
\end{equation*}%
where $\lambda =q-2\cos \theta $. From these last equations, we can select
pairs%
\begin{equation}
\frac{\alpha _{n+1}^{\prime \prime }}{-\alpha _{1}^{\prime }}=\frac{\alpha
_{1}^{\prime \prime }}{\alpha _{n+1}^{\prime }},...,\frac{\alpha
_{2n}^{\prime \prime }}{-\alpha _{n}^{\prime }}=\frac{\alpha _{n}^{\prime
\prime }}{\alpha _{2n}^{\prime }}.  \label{**}
\end{equation}%
Firstly, let $\lambda \neq 0$. Solving the ODEs, we have%
\begin{equation*}
\left( \alpha _{i}^{\prime }\right) ^{2}+\left( \alpha _{n+i}^{\prime
}\right) ^{2}=c_{i}^{2},\text{ }i=1,...,n
\end{equation*}%
for some arbitrary constants $c_{i}$ $(i=1,...,n)$ such that 
\begin{equation*}
\sum_{i=1}^{n}c_{i}^{2}=4\sin ^{2}\theta .
\end{equation*}%
So we have%
\begin{equation*}
\alpha _{i}^{\prime }=c_{i}\cos f_{i},\text{ }\alpha _{n+i}^{\prime
}=c_{i}\sin f_{i}
\end{equation*}%
for some differentiable functions $f_{i}:I\rightarrow 
\mathbb{R}
$ $(i=1,...,n)$. From (\ref{**}), we get%
\begin{equation*}
\frac{\alpha _{n+i}^{\prime \prime }}{-\alpha _{i}^{\prime }}=-f_{i}^{\prime
}=-\lambda ,
\end{equation*}%
which gives us%
\begin{equation*}
f_{i}=\lambda t+d_{i},
\end{equation*}%
for some arbitrary constants $d_{i}$ $(i=1,...,n)$. Here, $t$ denotes the
arc-length parameter. Then, we find 
\begin{equation*}
\alpha _{i}^{\prime }=c_{i}\cos \left( \lambda t+d_{i}\right) ,\text{ }%
\alpha _{n+i}^{\prime }=c_{i}\sin \left( \lambda t+d_{i}\right) .
\end{equation*}%
Finally, we obtain 
\begin{equation*}
\alpha _{i}=\frac{c_{i}}{\lambda }\sin \left( \lambda t+d_{i}\right) +h_{i},
\end{equation*}%
\begin{equation*}
\alpha _{n+i}=\frac{-c_{i}}{\lambda }\cos \left( \lambda t+d_{i}\right)
+h_{n+i},
\end{equation*}%
\begin{eqnarray*}
\alpha _{2n+1} &=&2t\cos \theta +\sum_{i=1}^{n}\left\{ \frac{-c_{i}^{2}}{%
4\lambda ^{2}}\left[ 2\left( \lambda t+d_{i}\right) +\sin \left( 2\left(
\lambda t+d_{i}\right) \right) \right] \right.  \\
&&\text{ \ \ \ \ \ \ \ \ \ \ \ \ \ \ \ \ \ \ \ }\left. +\frac{c_{i}h_{n+i}}{%
\lambda }\sin \left( \lambda t+d_{i}\right) \right\} +h_{2n+1},
\end{eqnarray*}%
for some arbitrary constants $h_{i}$ $(i=1,...,2n+1).$

Secondly, let $\lambda =0$. In this case, $q=2\cos \theta $ and $\widehat{k}%
_{1}=0$. Hence, we have%
\begin{equation*}
\widehat{\nabla }_{E_{1}}E_{1}=\frac{1}{2}\left( \sum_{i=1}^{n}\alpha
_{n+i}^{\prime \prime }X_{i}+\sum_{i=1}^{n}\alpha _{i}^{\prime \prime
}X_{n+i}\right) =0,
\end{equation*}%
which gives us%
\begin{equation*}
\alpha _{i}=c_{i}t+d_{i},\text{ }i=1,...,2n,
\end{equation*}%
\begin{equation*}
\alpha _{2n+1}=2t\cos \theta +\sum_{i=1}^{n}c_{i}\left( \frac{c_{n+i}}{2}%
t^{2}+d_{n+i}t\right) +c_{2n+1},
\end{equation*}%
where $c_{i}$ $(i=1,2,...,2n+1)$ and $d_{i}$ $(i=1,2,...,2n)$ are arbitrary
constants such that%
\begin{equation*}
\sum_{i=1}^{2n}c_{i}^{2}=4\sin ^{2}\theta .
\end{equation*}%
To conclude, we can state the following theorem:

\begin{theorem}
The pseudo-Hermitian magnetic curves on $%
\mathbb{R}
^{2n+1}(-3)$ endowed with the Tanaka-Webster connection have the parametric
equations 
\begin{equation*}
\alpha :I\subseteq 
\mathbb{R}
\rightarrow 
\mathbb{R}
^{2n+1}(-3),\text{ }\alpha =\left( \alpha _{1},\alpha _{2},...,\alpha
_{n},\alpha _{n+1},...,\alpha _{2n},\alpha _{2n+1}\right) ,
\end{equation*}%
where $\alpha _{i}$ $(i=1,...,2n+1)$ satisfies either

a) 
\begin{equation*}
\alpha _{i}=\frac{c_{i}}{\lambda }\sin \left( \lambda t+d_{i}\right) +h_{i},
\end{equation*}%
\begin{equation*}
\alpha _{n+i}=\frac{-c_{i}}{\lambda }\cos \left( \lambda t+d_{i}\right)
+h_{n+i},
\end{equation*}%
\begin{eqnarray*}
\alpha _{2n+1} &=&2\cos \theta t+\sum_{i=1}^{n}\left\{ \frac{-c_{i}^{2}}{%
4\lambda ^{2}}\left[ 2\left( \lambda t+d_{i}\right) +\sin \left( 2\left(
\lambda t+d_{i}\right) \right) \right] \right.  \\
&&\text{ \ \ \ \ \ \ \ \ \ \ \ \ \ \ \ \ \ \ \ }\left. +\frac{c_{i}h_{n+i}}{%
\lambda }\sin \left( \lambda t+d_{i}\right) \right\} +h_{2n+1},
\end{eqnarray*}%
where $\lambda =q-2\cos \theta \neq 0,$ $c_{i},$ $d_{i}$ $(i=1,...,n)$ and $%
h_{i}$ $(i=1,...,2n+1)$ are arbitrary constants such that%
\begin{equation*}
\sum_{i=1}^{n}c_{i}^{2}=4\sin ^{2}\theta ;
\end{equation*}%
or

b) 
\begin{equation*}
\alpha _{i}=c_{i}t+d_{i},
\end{equation*}%
\begin{equation*}
\alpha _{2n+1}=2t\cos \theta +\sum_{i=1}^{n}c_{i}\left( \frac{c_{n+i}}{2}%
t^{2}+d_{n+i}t\right) +c_{2n+1},
\end{equation*}%
where $q=2\cos \theta \ $and $c_{i}$ $(i=1,2,...,2n+1)$, $d_{i}$ $%
(i=1,2,...,2n)$ are arbitrary constants such that%
\begin{equation*}
q^{2}+\sum_{i=1}^{2n}c_{i}^{2}=4\text{.}
\end{equation*}
\end{theorem}

\end{document}